\documentclass[twoside,notitlepage,11pt]{article}

\usepackage{amssymb}
\usepackage[leqno]{amsmath}
\usepackage{amsfonts}
\usepackage{amsopn}
\usepackage{amstext}
\usepackage{amsthm}

\usepackage[all]{xy}
\newdir{ >}{{}*!/-9pt/@{>}}

\usepackage[all]{xy}
\usepackage{amscd}

\newcommand{\pr}{\mathrm{pr}}
\newcommand{\IN}{\mathbb N}
\newcommand{\e}{\varepsilon}

\newcommand{\eps}{\varepsilon }

\newcommand{\U}{\mathbb U}
\newcommand{\IU}{\mathbb U}
\newcommand{\K}{\mathfrak K}
\newcommand{\catB}{\mathfrak{B}}
\newcommand{\BI}{\mathfrak{BI}}
\newcommand{\DL}{\mathfrak L}
\newcommand{\DF}{\mathfrak{FI}_K}
\newcommand{\FI}{\mathfrak{FI}}
\newcommand{\RFI}{\mathfrak{RI}}
\newcommand{\RI}{\mathfrak{RI}}
\newcommand{\w}{\omega}

\newcommand{\bas}{\ss}
\newcommand{\conv}{\mathrm{conv}}

\newtheorem {theorem}{Theorem}[section]
\newtheorem{corollary}[theorem]{Corollary}
\newtheorem {lemma}[theorem]{Lemma}
\newtheorem{proposition}[theorem]{Proposition}

\theoremstyle{definition}
\newtheorem {definition}[theorem]{Definition}

\title{The universal Banach space with\\ a $K$-suppression unconditional basis}
\author{
{\sc Taras Banakh}\\ \\
{\small Ivan Franko University of Lviv (Ukraine) and Jan Kochanowski University (POLAND)}\\
{\small \texttt{t.o.banakh@gmail.com}} \\
{\sc Joanna Garbuli\'nska-W\c egrzyn}\thanks{Research of the second author was suppored by NCN grant DEC-2013/11/N/ST1/02963.} \\
{\small Institute of Mathematics, Jan Kochanowski University (POLAND)}\\
{\small \texttt{jgarbulinska@ujk.edu.pl}}
}
\begin{document}
\maketitle
\begin{abstract}
Using the technique of Fra\"iss\'e theory, for every constant $K\ge 1$ we construct a universal object $\mathbb U_K$ in the class of Banach spaces possessing a normalized $K$-suppression unconditional Schauder basis.


\end{abstract}
\section{Introduction}
A Banach space $X$ is \emph{complementably universal} for a given class of Banach spaces if $X$ belongs to this class and every space from the class is isomorphic to a complemented subspace of $X$.

In 1969 Pe\l czy\'nski \cite{pelbases} constructed a complementably universal Banach space for the class of Banach spaces with a Schauder basis.
In 1971 Kadec \cite{kadec} constructed a complementably universal Banach space for the class of spaces with the \emph{bounded approximation property} (BAP).
In the same year Pe\l czy\'nski \cite{pel_any} showed that every Banach space with BAP is complemented in a space with a basis.
Pe\l czy\'nski and Wojtaszczyk  \cite{wojtaszczyk} constructed in 1971 a universal Banach space for the class of spaces with a finite-dimensional decomposition.
Applying Pe\l czy\'nski's decomposition argument \cite{pelczynski}, one immediately concludes that all three universal spaces are isomorphic.
It is worth mentioning a negative result of Johnson and Szankowski \cite{JohnSzan} saying that no separable Banach space can be complementably universal for the class of all separable Banach spaces.
In \cite{asia} the second author constructed an isometric version of the Kadec-Pe\l czy\'nski-Wojtaszczyk space. The universal Banach space from \cite{asia} was constructed using the general categorical technique of Fra\"iss\'e limits \cite{kubis}. This method was also applied by Kubi\'s and Solecki in \cite{solecki} for constructing the Gurari\u\i\ space \cite{gurari}, which possesses the property of extension of almost isometries, which  implies the universality property that is stronger than the standard universality property of the Banach spaces $\ell _{\infty}$ or $C[0,1]$.

In this paper we apply the categorical method of Fra\"iss\'e limits for constructing a universal space $\IU_K$ in the class of Banach spaces with a normalized $K$-suppression unconditional Schauder basis. The universal space constructed by this method has a nice property of extension of almost isometries, which is better than just the standard universality, established in the papers of Pe\l czy\'nski \cite{pelbases} and Schechtman \cite{schechtman} (who gave a short alternative construction of universal space for class of Banach spaces with an unconditional bases). We also prove that the universal space $\IU_K$ is isomorphic to the complementably universal space $\IU$ for Banach spaces with unconditional basis, which was constructed by Pe\l czy\'nski in \cite{pelbases}.

\section{Preliminaries}
All Banach spaces considered in this paper are separable and over the field $\mathbb R$ of real numbers. 
\subsection{Definitions}

Let $X$ be a Banach space with a Schauder basis $(\mathsf e_n)_{n=1}^\infty$ and let $(\mathsf e_n^*)_{n=1}^\infty$ be the corresponding sequence of coordinate functionals.
The basis $(\mathsf e_n)_{n=1}^\infty$ is called {\em $K$-supression} for a real constant $K$ if for every finite subset $F\subset \mathbb N$ the projection $\pr_F:X\to X$, $\pr_F:x\mapsto \sum_{n\in F}\mathsf e^*_n(x)\cdot \mathsf e_n$, has norm $\|\pr_F\|\le K$.
It is well-known \cite[3.1.5]{AK} that each $K$-suppression Schauder basis $(\mathsf e_n)_{n=1}^\infty$ is unconditional. So for any $x\in X$ and any permutation $\pi$ of $\IN$ the series $\sum_{n=1}^\infty \mathsf e_{\pi(n)}^*(x)\cdot \mathsf e_{\pi(n)}$ converges to $x$. This means that we can forget about the ordering and think of a $K$-suppression basis of a Banach space as a subset $\bas\subset X$ such that for some bijection $\mathsf e:\mathbb N\to \bas$ the sequence $(\mathsf e(n))_{n=1}^\infty$ is a $K$-suppression Schauder basis for $X$.

More precisely, by a {\em normalized $K$-suppression basis} for a Banach space $X$ we shall understand a subset $\bas\subset X$ for which there exists a family $\{\mathsf e^*_b\}_{b\in \bas}\subset X$ of continuous functionals such that
\begin{itemize}
\item $\|b\|=1=\mathsf e^*_b(b)$ for any $b\in \bas$;
\item $\mathsf e_b^*(b')=0$ for every $b\in \bas$ and $b'\in \bas\setminus \{b\}$;
\item $x=\sum_{b\in \bas}\mathsf e^*_b(x)\cdot b$ for every $x\in X$;
\item for any finite subset $F\subset \bas$ the projection $\pr_F:X\to X$, $\pr_F:x\mapsto \sum_{b\in F}\mathsf e^*_b(x)\cdot b$, has norm $\|\pr_F\|\le K$.
\end{itemize}
The equality $x=\sum_{b\in \bas}\mathsf e^*_b(x)\cdot b$ in the third item means that for every $\e>0$ there exists a finite subset $F\subset \bas$ such that $\|x-\sum_{b\in E}\mathsf e^*_b(x)\cdot b\|<\e$ for every finite subset $E\subset \bas$ containing $F$.

By a \emph{$K$-based Banach space} we shall understand a pair $(X,\bas_X)$ consisting of a Banach space $X$ and a normalized $K$-suppression basis $\bas_X$ for $X$. By a {\em based Banach space} we understand a $K$-based Banach space for some $K\ge 1$. We shall say that a based Banach space $(X,\bas_X)$ is a \emph{subspace} of a based Banach space $(Y,\bas_Y)$ if $X\subseteq Y$ and $\bas_X=X\cap \bas_Y$.

For a Banach space $X$ by $\|\cdot\|_X$ we denote the norm of $X$ and by $B_X:=\{x\in X:\|x\|_X\le 1\}$ the closed unit ball of $X$.

A finite dimensional based Banach $(X,\bas_X)$ is called \emph{rational} if its unit ball $B_X$ is a convex polyhedron spanned by finitely many vectors with rational coordinates in the basis $\bas_X$. A based Banach space $X$ is called {\em rational} if each finite-dimensional based subspace of $X$ is rational.

\subsection{Categories}
Let $\K$ be a category. 
Morphisms and isomorphisms of a category $\K$ will be called {\em $\K$-morphisms} and {\em $\K$-isomorphisms}, respectively.
A \emph{subcategory} of $\K$ is a category $\DL$ such that each object of $\DL$ is an object of $\K$ and each morphism of $\DL$ is a morphism of $\K$.

A subcategory $\mathfrak L$ of a category $\K$ is {\em full} if each $\K$-morphism between objects of the category $\mathfrak L$ is an $\mathfrak L$-morphism.
A subcategory $\mathfrak L$ of a category $\K$ is \emph{cofinal} in $\K$ if for every object $A$ of $\K$ there exists a $\K$-morphism $f:A\to B$ to an object $B$ of $\mathfrak L$. 

A category $\K$ has the {\it amalgamation property} if for every objects $A, B, C$ of $\K$ and $\K$-morphisms $f:A\to B$ and $g:A\to C$ there exist an object $D$ of $\K$ and $\K$-morphisms $f':K\to D$ and $g':C\to D$ such that $f'\circ f=g'\circ g$.

In this paper we shall work in the category $\catB$, whose objects are based Banach spaces and morphisms are  linear continuous operators $T: X\to Y$ between based Banach spaces $(X,\bas_X)$ and $(Y,\bas_Y)$ such that $T(\bas_X)\subseteq \bas_Y$. 

A morphism $T:X\to Y$ of the category $\catB$ is called an \emph{isometry} (or else an {\em isometry morphism}) if $\|T(x)\|_Y=\|x\|_X$ for any $x\in X$.  By $\BI$ we denote the category whose objects are based Banach spaces and morphisms are isometry morphisms of based Banach spaces. The category $\BI$ is a subcategory of the category $\catB$. 

For any real number $K\ge 1$ let $\catB_K$ (resp. $\BI_K$) be the category whose objects are $K$-based Banach spaces and morphisms are (isometry) $\catB$-morphisms between $K$-based Banach spaces. So, $\catB_K$ and $\BI_K$ are full subcategories of the categories $\catB$ and $\BI$, respectively.  

By $\FI_K$ we denote the full subcategory of $\BI_K$, whose objects are finite-dimensional  $K$-based Banach spaces, and by $\RFI_K$ the full subcategory of $\FI_K$  whose objects are rational finite-dimensional  $K$-based Banach spaces.
So, we have the inclusions $\RFI_K\subset\FI_K\subset\BI_K$ of categories.

From now on we assume that $K\ge 1$ is some fixed real number.

\subsection{Amalgamation}
In this section we prove that the categories $\FI_K$ and $\RI_K$ have the amalgamation property.

\begin{lemma}\label{lemat} {\bf (Amalgamation Lemma)} Let $X$, $Y$, $Z$ be finite-dimensional $K$-based Banach spaces and $j:Z\to X$, $i:Z \to Y$ be $\BI$-morphisms. Then there exist a finite-dimensional $K$-based Banach space $W$ and $\BI$-morphisms $j':Y\to W$ and $i':X\to W$ such that  the diagram
$$\xymatrix{
Y \ar[r]^{j'} &W \\
Z \ar[r]^j \ar[u]^i & X\ar[u]^{i'}}$$
is commutative.\\
Moreover, if the $K$-based Banach spaces $X$, $Y$, $Z$ are rational, then so is the $K$-based Banach space $W$.
\end{lemma}

\begin{proof} We shall prove this lemma in the special case when the isometries $i$, $j$ are identity inclusions; the general case is analogous but has more complicated notation. Our assumptions on $i,j$ ensure that $Z=X\cap Y$ and $\bas_Z=\bas_X\cap \bas_Y$, where $\bas_X,\bas_Y,\bas_Z$ are the normalized $K$-suppression bases of the $K$-based Banach spaces $X,Y,Z$. It follows from $\bas_Z=\bas_X\cap \bas_Y$ that the coordinate functionals of the bases $\bas_X$ and $\bas_Y$ agree on the intersection $Z=X\cap Y$.

Consider the direct sum $X\oplus Y$ of the Banach space $X$, $Y$ endowed with the norm $\| (x,y)\| =\| x\| _X+\| y\| _Y$.
Let $W=(X\oplus Y)/\Delta $ be the quotient space of $X\oplus Y$ by the subspace $\Delta =\{(z,-z):z\in Z\}$.

We define linear operators $i':X\to W$ and $j':Y\to W$ by $i'(x)=(x,0)+\Delta $ and $j'(y)=(0,y)+\Delta .$\\
Let us show $i'$ and $j'$ are isometries. Indeed, for every $x\in X$
\begin{align*}
\| i'(x)\|_W &=\operatorname{dist}((x,0),\Delta )\leq \| (x,0)\| =\| x\| _X+\| 0\| _Y=\| x\| _X.
\end{align*}
On the other hand, for every $z\in Z$
\begin{align*}
&\| (x,0)-(z,-z)\| =\| (x-z,z)\| =\| x-z\| _X+\| z\| _Y=\\&=\|  x-z\| _X+\| z\| _X\geq \| x-z+z\| _X=\| x\| _X
\end{align*}
and hence $\|x\|_X \leq \inf_{z\in Z} \| (x,0)-(z,-z)\|=\|i'(x)\|_W$. Therefore $\|i'(x)\|_W=\|x\|_X$. 
Similarly, we can show that $j'$ is an isometry.

We shall identify $X$ and $Y$ with their images $i'(X)$ and $j'(Y)$ in $W$.
In this case we can consider the union $\bas_W:=\bas_X\cup \bas_Y$ and can show that $\bas_W$ is a normalized Schauder basis for the (finite-dimensional)  Banach space $W$.  Let $\{\mathsf e_b^*\}_{b\in \bas_W}\subset W^*$ be the sequence of coordinate functionals of the basis $\bas_W$.

Let us show that the basis $\bas_W$ is $K$-suppression. Given any subset $D$ of $\bas_W$ we should prove that the projection $\pr_D:W\to W$, $\pr_D:w\mapsto \sum_{b\in D}\mathsf e^*_b(w)b$, has norm $\|\pr_D\|\le K$.

Write the set $D$ as $D=D_Z\cup D_X\cup D_Y$, where $D_Z=D\cap \bas_Z=D\cap \bas_X\cap \bas_Y$, $D_X=D\backslash \bas_Y$ and $D_Y=D\backslash \bas_X$.

 Taking into account that the bases $\bas_X$ and $\bas_Y$ are $K$-suppression, for any $w\in W$ we obtain:
\begin{align*}
&\|\pr_D(w)\|_W=
\inf\{\|x\|_X+\|y\|_Y:x\in X,\;y\in Y,\;x+y=\pr_D(w)\}=\\
&=\inf\{\|\pr_{D_X}(w)+z'\|_X+\|z''+\pr_{D_Y}(w)\|_Y:z',z''\in Z,\;z'+z''=\pr_{D_Z}(w)\}\le\\
&\le\inf\{\|\pr_{D_X}(w)+z'\|_X+\|z''+\pr_{D_Y}(w)\|_Y:z',z''\in \pr_{D_Z}(Z),\;z'+z''=\pr_{D_Z}(w)\}=\\
&=\inf\{\|\pr_{D_X}(w)+\pr_{D_Z}(z')\|_X+\|\pr_{D_Z}(z'')+\pr_{D_Y}(w)\|_Y:z',z''\in Z,\;z'+z''=\pr_{\bas_Z}(w)\}\le\\
&\le K\cdot \inf\{\|\pr_{\bas_X\setminus \bas_Z}(w)+z'\|_X+\|z''+\pr_{\bas_Y\setminus \bas_X}(w)\|_Y:\; z',z''\in Z,\;\;z'+z''=\pr_{\bas_Z}(w)\}=\\
&=K\cdot \inf\{\|x\|_X+\|y\|_Y:x+y=w\}=K\cdot\|w\|_W.
\end{align*}

If the finite-dimensional based Banach spaces $X$ and $Y$ are rational, then so is their sum $X\oplus Y$ and so is the quotient space $W$ of $X\oplus Y$.
\end{proof}

\section{$\catB$-universal based Banach spaces}

\begin{definition}\label{d:uB} A based Banach space $U$ is defined to be {\em $\catB$-universal} if each based Banach space $X$ is $\catB$-isomorphic to a based subspace of $U$.
\end{definition}

Definition~\ref{d:uB} implies that each $\catB$-universal based Banach space is complementably universal for the class of Banach spaces with unconditional basis.
 Reformulating Pe\l czy\'nski's Uniqueness Theorem 3 \cite{pelbases}, we obtain the following uniqueness result.

\begin{theorem}[Pe\l czy\'nski]\label{t:up} Any two $\catB$-universal based Banach spaces are $\catB$-isomorphic.
\end{theorem}

A $\catB$-universal based Banach space $\IU$ was constructed by Pe\l czy\'nski in \cite{pelbases}. In the following sections we shall apply the technique of Fra\"iss\'e limits to construct many $\catB$-isomorphic copies of the Pe\l czy\'nski's $\catB$-universal space $\U$.

\section{$\RI_K$-universal based Banach spaces}\label{s:RI}


\begin{definition}\label{def1}
A based Banach space $X$ is called \emph{$\RI_K$-universal} if for any rational finite-dimensional $K$-based Banach space $A$, any isometry morphism $f:\Lambda\to X$ defined on a based subspace $\Lambda$ of $A$ can be extended to an isometry morphism $\bar f:A\to X $.
\end{definition}

We recall that $\RI_K$ denotes the full subcategory of $\BI$ whose objects are rational finite-dimensional $K$-based Banach spaces.
Obviously, up to isomorphism the category $\RI_K$ contains countably many objects.
By Lemma~\ref{lemat}, the category $\RI_K$ has the amalgamation property.
We now use the concepts from \cite{kubis} for constructing a ``generic" sequence in $\RI_K$.
A sequence $(X_n)_{n\in \omega}$ of objects of the category $\BI_K$ is called \emph{a chain} if each $K$-based Banach space $X_n$ is a subspace of the $K$-based Banach space $X_{n+1}$.  

\begin{definition}\label{Fresse}
A chain of $(U_n)_{n\in \omega}$ of objects of the category $\RI_K$ is \emph{Fra\"iss\'e} if for any $n\in \omega$ and $\RI_K$-morphism $ f: {U_n} \to Y$ there exist $m > n$ and an $\RI_K$-morphism  $g: Y\to {U_m}$ such that $g \circ f:U_n\to U_m$ is the identity inclusion of $U_n$ to $U_m$.
\end{definition}

Definition \ref{Fresse} implies that the Fra\"iss\'e sequence $\{U_n\}_{n\in \omega}$ is cofinal in the category $\RI_K$ in the sense that each object $A$ of the category $\DF $ admits an $\RI_K$-morphism $A\to U_n$ for some $n\in \omega$. This means that the category $\RI_K$ is countably cofinal.

The name ``Fra\"iss\'e sequence", as in \cite{kubis}, is motivated by the model-theoretic theory of Fra\"iss\'e limits developed by Roland Fra\"iss\'e \cite{fra}.
One of the results in \cite{kubis} is that every countably cofinal category with amalgamation has a Fra\"iss\'e sequence.
Applying this general result to our category $\RI_K$ we get:

\begin{theorem}[\cite{kubis}, Theorem 3.7]\label{Kubisia}
The category $\RI_K$ has a Fra\"iss\'e sequence.
\end{theorem}

From now on, we fix a Fra\"iss\'e\ sequence $(U_n)_{n\in \omega}$ in $\RI_K$, which can be assumed to be a chain of finite-dimensional rational $K$-based Banach spaces.
Let $\U_K$ be the completion of the union $\bigcup_{n\in \omega}U_n$ and $\bas_{\U_K}=\bigcup_{n\in \omega}\bas_{U_n}\subset\U_K$.

\begin{theorem} The pair 
$(\U_K,\bas_{\U_K})$ is an $\RI_K$-universal rational $K$-based Banach space.
\end{theorem}

\begin{proof} First we show that $\bas_{\U_K}=\bigcup_{n\in \omega}\bas_{U_n}$ is a normalized $K$-suppression basis for $\U_K$.
The fact that $\bas_{\U_K}$ is an unconditional Schauder basis with suppression constant $K$ follows from Lemma 6.2 and Fact 6.3 in \cite{fabian}.
For each $n$ the spaces $U_n$ are $K$-based Banach spaces, so $\|b\|=1$ for every $b\in \bas_{U_n}$. This shows that the basis $\bas_{\U_K}$ is normalized.

The based Banach space $(\U_K,\bas_{\U_K})$ is rational, since each finite-dimensional based subspace of $(\U_K,\bas_{\U_K})$ is contained in some rational based Banach space $(U_n,\bas_{U_n})$ and hence is rational.  

The $\RI_K$-universality of the based Banach space $(\U_K,\bas_{U_K})$ follows from the construction and \cite[Proposition 3.1]{kubis}.
\end{proof}

To shorten notation, the $\RI_K$-universal rational $K$-based Banach space $(\U_K,\bas_{\U_K})$ will be denoted by $\U_K$. The following theorem shows that such space is unique up to $\BI$-isomorphism.

\begin{theorem}\label{t:BI}
Any $\RI_K$-universal rational $K$-based Banach spaces $X$, $Y$ are $\BI$-isomorphic, which means that there exists a linear bijective isometry $X\to Y$ preserving the bases of $X$ and $Y$.
\end{theorem}

\begin{proof}  By definition, the rational $K$-based Banach spaces $X,Y$ can be written as the completions of unions $\bigcup_{n\in \omega} X_n$ and $\bigcup_{n\in \omega} Y_n$ of chains $(X_n)_{n\in\w}$ and $(Y_n)_{n\in\w}$ of rational finite-dimensional $K$-based Banach spaces such that $X_0=\{0\}$ and $Y_0=\{0\}$ are trivial $K$-based Banach spaces. 

We define inductively sequences of $\RI_K$-morphisms $\{f_k\}_{k\in \omega}$, $\{g_k\}_{k\in \omega}$ and increasing number sequences $(n_k)$, $(m_k)$ such that the following conditions are satisfied for every $k\in \omega$:
\begin{enumerate}\itemsep=2pt\parskip=2pt
    \item[(1)] $f_{k}:X_{n_{k-1}}\rightarrow  Y_{m_k}$ and $g_k:Y_{m_k}\rightarrow  X_{n_{k}}$ are morphisms of category $\mathfrak{RI}_K$;
    \item[(2)] $f_{k+1}\circ g_{k}=\operatorname{id}\restriction Y_{m_k}$ and $g_{k+1}\circ f_{k+1}=\operatorname{id}\restriction X_{n_k}$.
\end{enumerate}
We start the inductive construction letting $n_0=0=m_0$ and $f_0:X_0\to Y_0$,  $g_0:Y_0\to X_0$ be the unique isomorphisms of the trivial $K$-based Banach spaces $X_0$ and $Y_0$. To make an inductive step, assume that for some $k \in \omega$, the numbers $n_k$, $m_k$ and $\RI_K$-morphisms $f_k:X_{n_{k-1}}\rightarrow  Y_{m_k}$, $g_k:Y_{m_k}\rightarrow  X_{n_{k}}$ have been constructed. By Definition \ref{def1}, the $\BI$-morphism $g_k^{-1}:g_k(Y_{m_k})\to Y$ defined on the based subspace $g_k(Y_{m_k})$ of the rational finite-dimensional $K$-based Banach space $X_{n_k}$ extends to a $\BI$-morphism $f_{k+1}:X_{n_k}\to Y$. So, $f_{k+1}\circ g_k=\operatorname{id}\restriction Y_{m_k}$. Since $f_{k+1}(\bas_{X_{n_k}})\subset \bas_Y=\bigcup_{i\in\w}\bas_{Y_i}$, there exists a number $m_{k+1}$ such that $f_{k+1}(\bas_{X_{n_k}})\subset\bas_{Y_{m_{k+1}}}$ and hence $f_{k+1}(X_{n_k})\subset Y_{m_{k+1}}$. Since the based space $Y$ is rational, its based subspace $Y_{m_{k+1}}$ is an object of the category $\RI_K$ and the morphism $f_{k+1}:X_{n_k}\to Y_{m_{k+1}}$ is an $\RI_K$-morphism. 

By analogy we can use the $\RI_K$-universality of the based Banach space $X$  find a number $n_{k+1}> n_k$ and an $\RI_K$-morphism $g_{k+1}:Y_{m_{k+1}}\to X_{n_{k+1}}$ such that $g_{k+1}\circ f_{k+1}$ is the identity inclusion $X_{n_k}$ in $X_{n_{k+1}}$.
This complete the inductive step.

After completing the inductive construction consider the isometries $f:\bigcup_{n\in \omega} X_n \to \bigcup_{m\in \omega} Y_m$ and $g: \bigcup_{m\in \omega} Y_m \to \bigcup_{n\in \omega} X_n$ such that $f\restriction X_{n_k}=f_{k+1}$ and $g\restriction Y_{m_k}=g_{k}$  for every $k\in \omega$.

By the uniform continuity, the isometries $f$, $g$ extend to isometries $\bar{f}:X\to Y$ and $\bar{g}:Y\to X$.

The condition (2) of the inductive construction implies that $\bar{f}\circ \bar g =\operatorname{id}_Y$ and $\bar g\circ \bar f=\operatorname{id}_ X$, so $f$ and $g$ are isometric isomorphisms of the Banach spaces $X$ and $Y$. Since the isometries $g_k:Y_{m_k}\to X_{n_k}$ are morphisms of based Banach spaces, we  get $$g(\bas_Y)=g\big(\bigcup_{k\in\w}\bas_{Y_{m_k}}\big)=\bigcup_{k\in\w}g(\bas_{Y_{m_k}})=\bigcup_{k\in\w}g_k(\bas_{Y_{m_k}})\subset
\bigcup_{k\in\w}\bas_{X_{n_k}}=\bas_{X}.$$
By analogy we can show that $f(\bas_X)\subset\bas_Y$. So, $f$ and $g$ are $\BI$-isomorphisms.
\end{proof}

\section{Almost $\FI_K$-universality}

By analogy with the $\RI_K$-universal based Banach space, one can try to introduce a $\FI_K$-universal based Banach space. However such notion is vacuous as each based Banach space has only countably many finite-dimensional based subspaces whereas the category $\FI_K$ contains continuum many pairwise non $\BI$-isomorphic  2-dimensional based Banach spaces. A ``right'' definition is that of an almost $\FI_K$-universal based Banach space, introduced with the help of $\e$-isometries.

For a positive real number $\e$, a linear operator $f:X\to Y$ between Banach spaces $X$ and $Y$ is called an \emph{$\eps$-isometry}  if
$$ (1+\eps )^{-1} \cdot \|x\|_X <  \|f(x)\|_Y < (1+\eps ) \cdot \|x\|_X$$
for every $x\in X\backslash \{0\}$. This definition implies that each $\e$-isometry is an injective linear operator.

A morphism  of the category $\catB$ of based Banach spaces is called an {\em $\e$-isometry $\catB$-morphism} if it is an $\e$-isometry of the underlying Banach spaces.

\begin{definition}\label{def2}
A based Banach space $X$ called \emph{almost $\FI_K$-universal} if for any $\eps>0$ and finite dimensional $K$-based Banach space $A$, any $\eps$-isometry $\catB$-morphism $f:\Lambda\to X$ defined on a based subspace $\Lambda$ of $A$ can be extended to a $\eps$-isometry $\catB$-morphism $\bar f:A\to X $.
\end{definition}

Unfortunately, this notion vacuous for $K>1$ as shown in the following proposition that can be proved by analogy with Proposition 5.8 in \cite{ataras}. 

\begin{proposition}\label{p:non} No based Banach space is almost $\FI_K$-universal for $K>1$.
\end{proposition}

\begin{theorem}\label{t:ru=>au}
Any $\RI_K$-universal rational $K$-based Banach space $X$ is almost $\FI_1$-universal.
\end{theorem}

\begin{proof}
We shall use the fact, that the norm of any finite-dimensional based Banach space can be approximated by a rational norm (which means that its unit ball coincides with the convex hull of finitely many points having rational coordinates in the basis).

Let X be an $\RI_K$-universal rational K-based Banach space for some $K\geq 1$.
To prove that $X$ is almost $\FI_1$-universal,  take any $\eps>0$, any finite-dimensional $1$-based Banach space $A$ and an $\eps$-isometry $\catB$-morphism $f:\Lambda\to X$ defined on a based subspace $\Lambda$ of $A$. We recall that by $\|\cdot\|_A$ and $\|\cdot\|_\Lambda$ we denote the norms of the Banach spaces $A$ and $\Lambda$. 
The morphism $f$ determines a new norm $\| \cdot \|_\Lambda'$ on $\Lambda$, defined by $\| a \|_\Lambda'=\|f(a)\|_X$ for $a\in \Lambda$. Since $X$ is rational and $K$-based, $\|\cdot \|'_\Lambda$ is a  rational norm on $\Lambda$ such that $\|\pr_F(a)\|'_\Lambda\le K\cdot \|a\|'_\Lambda$ for every $a\in \Lambda$ and every subset $F\subset \bas_\Lambda$. Taking into account that $f$ is an $\eps$-isometry, we conclude that $(1+\eps)^{-1}<\|a\|_\Lambda'<(1+\eps)$ for every $a\in \Lambda$ with $\|a\|_\Lambda=1$. By the compactness of the unite sphere in $\Lambda$, there exists a positive $\delta <\eps$ such that $(1+\delta )^{-1}<\|a\|'_\Lambda<(1+\delta )$ for every $a\in \Lambda$ with $\|a\|_\Lambda=1$.  This inequality implies $\frac{1}{1+\delta }B_\Lambda \subset B_\Lambda'\subset (1+\delta )B_\Lambda$, where $B_\Lambda=\{a\in \Lambda: \|a\|_\Lambda\leq 1\}$ and $B_\Lambda'=\{a\in \Lambda: \|a\|'_\Lambda\leq 1\}$ are the closed unit balls of $\Lambda$ in the norms $\| \cdot \|_\Lambda$ and $\|\cdot \|_\Lambda'$. Choose $\delta '$ such that  $\delta< \delta' <\eps$. 

Let $B_A=\{a\in A: \|x\|_A\leq 1\}$ be the closed unit ball of the Banach space $A$. Choose a rational polyhedron $P$ in $A$ such that $P=-P$ and $\frac{1}{1+\delta' }B_A\subset P \subset \frac{1}{1+\delta }B_A$. 
Since $A$ is a 1-based Banach space, $\bigcup_{F\subset\bas_A}\pr_F(B_A)=B_A$. So, we can replace $P$ by the convex hull of $\bigcup_{F\subset\bas_A}\pr_F(P)$ and assume that $\pr_F(P)=P$ for every $F\subset\bas_A$.

Consider the set
$$P':=B'_\Lambda\cup\bas_A\cup P.$$ The convex hull $B_A':=\operatorname{conv} (P')$ of $P'$ is a rational polyhedron in the based Banach space $A$. The rational polyhedron $B_A'$, being convex and symmetric, determines a rational norm $\| \cdot\|_A'$ on $A$ whose closed unit ball coincides with $B_A'$. By $A'$ we denote the Banach space $A$ endowed with the norm $\|\cdot\|'_A$.

Taking into account that $P\subset\frac1{1+\delta}B_A$ and $A$ is a $1$-based Banach space, we conclude that 
\begin{multline*}
P'\subset B_\Lambda'\cup \tfrac1{1+\delta}\big(B_A\cup \bigcup_{F\subset \bas_{A}} \pr_F(B_A)\big)=B'_\Lambda\cup\tfrac1{1+\delta}(B_A\cup B_A)=\\
=B'_\Lambda\cup\tfrac1{1+\delta}B_A\subset(1+\delta)B_\Lambda\cup B_A\subset(1+\delta)B_A
\end{multline*}

Let us show that $\|a\|'_A=\|a\|'_\Lambda$ for each $a\in \Lambda$, which is equivalent to the equality $B_A'\cap \Lambda=B_\Lambda'$. The inclusion $B_\Lambda'\subset B_A'\cap \Lambda$ is evident.
To prove the reverse inclusion $B_\Lambda'\supset B_A'\cap \Lambda$, consider the coordinate projection $$\pr_\Lambda:A\to \Lambda,\;\;\textstyle{\pr_\Lambda:\sum_{b\in\bas_A}x_bb\mapsto \sum_{b\in\bas_\Lambda}x_bb}$$onto the subspace $\Lambda$ of $A$. Taking into account that $A$ is a 1-based Banach space, we conclude that $\pr_\Lambda(B_A)=B_A\cap\Lambda=B_\Lambda$. 
Then $$
\begin{aligned}
\pr_\Lambda(P')&=\pr_\Lambda(B_\Lambda')\cup\pr_\Lambda(\bas_A)\cup\pr_\Lambda(P)\subset\\
&\subset B'_\Lambda\cup\{0\}\cup\bas_\Lambda\cup\pr_\Lambda(\tfrac1{1+\delta}B_A)=B'_\Lambda\cup \tfrac1{1+\delta}B_\Lambda=B'_\Lambda
\end{aligned}
$$as $\frac1{1+\delta}B_\Lambda\subset B_\Lambda'$. Consequently, $$B'_A\cap\Lambda\subset \pr_\Lambda(B'_A)=\pr_\Lambda(\conv(P'))=\conv(\pr_\Lambda(P'))\subset\conv(B'_\Lambda)=B_\Lambda',$$which completes the proof of the equality $B_A'\cap \Lambda=B'_\Lambda$.

The inclusion $\bas_A\subset B_A'$ implies that $\|b\|'_A\le 1$ for any $b\in\bas_A$. We claim that $\|b\|_A'=1$ for any $b\in\bas_A$. If $b\in\bas_\Lambda$, then 
$f(b)\in\bas_X$ and $\|f(b)\|_X=1$. By the Hahn-Banach Theorem, there exists a linear continuous functional $x^*\in X^*$ such that $x^*(f(b))=1$ and $x^*(B_X)\subset[-1,1]$ where $B_X=\{x\in X:\|x\|_X\le 1\}$ is the closed unit ball of the Banach space $X$. Now consider the linear functional $a^*=x^*\circ f\circ \pr_\Lambda\in A^*$ and observe that $$a^*(B_A')=x^*\circ f(\pr_\Lambda(B_A'))=x^*\circ f(B'_\Lambda)\subset x^*(B_X)\subset[-1,1],$$which means that $a^*$ has norm $\|a^*\|'_{A^*}=1$ in the dual Banach space $(A')^*$. Now we see that $1=x^*(f(b))=a^*(b)\le\|a^*\|'_{A^*}\cdot\|b\|'_{A}\le  \|b\|'_{A}$ and hence $\|b\|'_{A}=1$.

If $b\notin \bas_\Lambda$, then we can  consider the coordinate functional $\mathbf e^*_b\in A^*$ of $b$.
Since $\bas_A$ is a $1$-suppression basis for the 1-based Banach space $A$, $\mathbf e^*_b(B_A)\subset[-1,1]$.  Then $$\mathbf e^*_b(P')\subset \mathbf e^*_b(B'_\Lambda)\cup\mathbf e^*_b(\pm\bas_A)\cup \mathbf e^*_b(B_A)\subset  \{0\}\cup \{-1,0,1\}\cup [-1,1]=[-1,1],$$ which means that the functional $\mathbf e^*_b$ has norm $\|\mathbf e^*_b\|'_{A^*}\le 1$ in the dual Banach space $(A')^*$.
Then $1=\mathbf e^*_b(b)\le\|\mathbf e^*_b\|'_{A^*}\cdot\|b\|'_{A}\le \|b\|'_{A}$ and hence $\|b\|'_{A}=1$. Therefore, the Banach space $A'$ endowed with the base $\bas_{A'}:=\bas_A$ is a based Banach space.

Next, we show that the based Banach space $A'$ is $K$-based. Indeed, for any $F\subset \bas_{A}$ we get 
$$
\pr_F(P')=\pr_F(B'_\Lambda)\cup \pr_F(\bas_A)\cup \pr_F(P)\subset K\cdot B'_\Lambda\cup(\{0\}\cup \bas _A)\cup P=[-K,K]\cdot P'$$ and hence 
\begin{multline*}
\pr_F(B'_A)=\pr_F(\conv(P'))=\conv(\pr_F(P'))\subset \conv([-K,K]\cdot P')=\\
=[-K,K]\cdot\conv(P')=K\cdot B_A',
\end{multline*} witnessing that the based Banach space $A'$ is $K$-based.

The inclusions $\frac{1}{1+\delta' }B_A\subset B_A'\subset (1+\delta )B_A$ imply the strict inequality 
\begin{equation}\label{eq1}
(1+\eps)^{-1}\|a\|_A < \|a\|_A' <(1+\eps)\|a\|_A
\end{equation} holding for all $a\in A\backslash \{0\}$.

Let $\Lambda'$ and $A'$ be the $K$-based Banach spaces $\Lambda$ and $A$ endowed with the new rational norms $\|\cdot\|_\Lambda'$ and $\|\cdot\|_A'$, respectively. It is clear that $\Lambda'\subset A'$. The definition of the norm $\|\cdot\|'_\Lambda$ ensures that $f:\Lambda'\to X$ is a $\BI$-morphism.
Using the $\RI_K$-universality of $X$, extend the isometry morphism $f:\Lambda'\to X$ to an isometry morphism $\bar f:A'\to X$. The inequalities~\eqref{eq1} ensure that $\bar f:A\to X$ is an $\eps$-isometry $\catB$-morphism from $A$, extending the $\eps$-isometry $f$.
This completes the proof of the almost $\FI_1$-universality of $X$.

Next consider the convex hull $B_A':=\operatorname{conv} (P')$ of the set $P'=B_\Lambda'\cup P\cup \bigcup_{F\subset \bas_{A}}\pr_F(P)$ and observe that $B_A'$ is a rational polyhedron in the based Banach space $A$. Taking into account that $P\subset\frac1{1+\delta}B_A$, $B_\Lambda'\subset(1+\delta)B_\Lambda\subset(1+\delta)B_A$,  and $A$ is a $1$-based Banach space, we conclude that 
\begin{multline*}
P'\subset B_\Lambda'\cup \tfrac1{1+\delta}\big(B_A\cup \bigcup_{F\subset \bas_{A}}\pr_F(B_A)\big)=B'_\Lambda\cup\tfrac1{1+\delta}(B_A\cup B_A)=\\
=B'_\Lambda\cup\tfrac1{1+\delta}B_A\subset(1+\delta)B_\Lambda\cup\tfrac1{1+\delta}(1+\delta)B_A\subset(1+\delta)B_A
\end{multline*}

and hence
$$\frac{1}{1+\delta'}B_A\subset P\subset  B_A':=\operatorname{conv} (P')\subset (1+\delta )B_A.$$
The convex symmetric set $B_A':=\conv(P')$ determines a rational norm $\| \cdot\|_A'$ on $A$ whose unite ball coincides with $B_A'$. We claim that the base $\bas_{A}$ of the Banach space $A':=(A,\|\cdot\|_A')$ is $1$-suppression.
Indeed, for any set $F\subset\bas_A$ we have 
$$\pr_F(P')=\pr_F(B_\Lambda')\cup\pr_F(P)\cup\bigcup_{E\subset\bas_A}\pr_F\circ\pr_E(P)\subset B'_\Lambda\cup P'\cup P'\subset P'$$and hence $$\pr_F(B_A')=\pr_F(\conv(P'))=\conv(\pr_F(P'))\subset \conv(P')=\conv(P')=B_A',$$which means that the projection $\pr_F:A'\to A'$ has norm $\le 1$ and $A'$ is a $1$-based Banach space.

It remains to check that $\|a\|'_A=\|a\|'_\Lambda$ for each $a\in \Lambda$, which is equivalent to the equality $B_A'\cap \Lambda=B_\Lambda'$. The inclusion $B_\Lambda'\subset B_A'\cap \Lambda$ is evident.
To prove the reverse inclusion $B_\Lambda'\supset B_A'\cap \Lambda$ observe that
\begin{align*}
\Lambda\cap B_A'&=\Lambda\cap \operatorname{conv} (P') \subset \Lambda\cap\mathrm{conv}(B_\Lambda'\cup \tfrac1{1+\delta}B_A)=\\
&=\Lambda\cap \{ t\lambda+(1-t)a:t\in[0,1],\; \lambda\in B'_\Lambda,\; a\in \tfrac1{1+\delta}B_A\}=\\&
=\{ t\lambda+(1-t)a: t\in[0,1],\;\lambda\in B_\Lambda',\;a\in \tfrac1{1+\delta}(\Lambda\cap B_A)\}\subset\\& \subset\operatorname{conv} (B_\Lambda'\cup B_\Lambda')=B_\Lambda'.
\end{align*}
The inclusions $\frac{1}{1+\delta' }B_A\subset B_A'\subset (1+\delta )B_A$ imply the strict inequality 
\begin{equation}\label{eq1}
(1+\eps)^{-1}\|a\|_A < \|a\|_A' <(1+\eps)\|a\|_A
\end{equation} holding for all $a\in A\backslash \{0\}$.

Let $\Lambda'$ and $A'$ be the $1$-based Banach spaces $\Lambda$ and $A$ endowed with the new rational norms $\|\cdot\|_\Lambda'$ and $\|\cdot\|_A'$, respectively. It is clear that $\Lambda'\subset A'$. The definition of the norm $\|\cdot\|'_\Lambda$ ensures that $f:\Lambda'\to X$ is a $\BI$-morphism.
Using the $\RI_K$-universality of $X$, extend the isometry morphism $f:\Lambda'\to X$ to an isometry morphism $\bar f:A'\to X$. The inequalities~\eqref{eq1} ensure that $\bar f:A\to X$ is an $\eps$-isometry $\catB$-morphism from $A$, extending the $\eps$-isometry $f$.
This completes the proof of the almost $\FI_1$-universality of $X$.
\end{proof}

\begin{theorem}\label{glowne}
Let $X$ and $Y$ be almost $\FI_K$-universal $K$-based Banach spaces and $ \eps >0$. Each $\eps$-isometry $\catB$-morphism $f:X_0\rightarrow Y$ defined on a finite-dimensional based subspace $X_0$ of the $K$-based Banach space $X$ can be extended to an $\eps$-isometry $\catB$-isomorphism $\bar{f}:X \rightarrow Y$.
\end{theorem}

\begin{proof}
Fix a positive real number $\eps$. Using the compactness of the unite sphere of the finite dimensional Banach space $X_0$, we can find a positive $\delta <\eps$ such that $f$ is a $\delta $-isometry. Write $X$ and $Y$ as the completions of the unions $\bigcup_{n\in \omega} X_n$ and $\bigcup_{n\in \omega} Y_n$ of chains of finite dimensional $K$-based Banach spaces such that $Y_0=f(X_0)$.
We define inductively sequences of $\catB$-morphisms $\{f_k\}_{k\in \omega}$, $\{g_k\}_{k\in \omega}$ and increasing number sequences $(n_k)$, $(m_k)$ such that $n_0=m_0=0$, $f_0=f$ and the following conditions are satisfied for every $k\in \omega$:
\begin{enumerate}
    \item[(1)] $f_{k}:X_{n_{k-1}}\rightarrow  Y_{m_k}$ and $g_k:Y_{m_k}\rightarrow  X_{n_{k}}$ are $\delta $-isometry $\catB$-morphisms;
    \item[(2)] $f_{k+1}\circ g_{k}=\operatorname{id}\restriction Y_{m_k}$ and $g_{k+1}\circ f_{k+1}=\operatorname{id}\restriction X_{n_k}$.
\end{enumerate}
To make the inductive step assume that for some $k \in \omega$, the numbers $n_k$, $m_k$ and  $\delta $-isometries $f_k:X_{n_{k-1}}\rightarrow  Y_{m_k}$, $g_k:Y_{m_k}\rightarrow  X_{n_{k}}$ have been constructed. Definition \ref{def2} of almost $\FI_K$-universality of the based Banach space $Y$ yields a 
 $\delta $-isometry $\catB$-morphism $f_{k+1}:X_{n_k}\to Y$ such that $f_{k+1}|g_k(Y_{m_k})=g_k^{-1}|g_k(Y_{m_k})$ and hence $f_{k+1}\circ g_{k}=\operatorname{id}\restriction Y_{m_k}$. Since $f_{k+1}(\bas_{X_{n_k}})$ is a finite subset of the basis $\bas_Y=\bigcup_{i\in\w}\bas_{Y_i}$ of $Y$, there exists a number $m_{k+1}> m_k$ such that $f_{k+1}(\bas_{X_{n_k}})\subset \bas_{Y_{m_{k+1}}}$ and hence $f_{k+1}(X_{n_k})\subset Y_{m_{k+1}}$.
 
By analogy, we can use the almost $\FI_K$-universality of the based Banach space $X$ and find a number $n_{k+1}> n_k$ and a $\delta $-isometry $\catB$-morphism $g_{k+1}:Y_{m_{k+1}}\to X_{n_{k+1}}$ such that $g_{k+1}\circ f_{k+1}=\operatorname{id}\restriction X_{n_k}$.
This complete the inductive step.

After completing the inductive construction consider the  $\delta $-isometries $\tilde{f}:\bigcup_{n\in \omega} X_n \to \bigcup_{m\in \omega} Y_m$ and $\tilde{g}: \bigcup_{m\in \omega} Y_m \to \bigcup_{n\in \omega} X_n$ such that for every $k\in \omega$ $\tilde{f}\restriction X_{n_k}=f_{k+1}$ and $\tilde{g}\restriction Y_{m_k}=g_{k}$. The condition (2) of the inductive construction implies that $\tilde{f}\circ \tilde{g}$ and $\tilde{g}\circ \tilde{f}$ are the identity maps of $\bigcup_{n\in \omega} X_n$ and $\bigcup_{m\in \omega} Y_m$, respectively.

By the uniform continuity, the  $\delta $-isometries $\tilde{f}$, $\tilde{g}$ extend to $\eps$-isometries $\bar{f}:X \to Y$ and $\bar{g}:Y \to X$ such that $\bar{f}\circ \bar{g}=\operatorname{id}_Y $ and $\bar{g}\circ \bar{f}=\operatorname{id}_X $. Taking into account that $f_n$ and $g_n$ are $\catB$-morphisms, we can show (repeating the argument from the proof of Theorem~\ref{t:BI}) that  the operators $\tilde f$ and $\tilde g$ preserve the bases of the $K$-based Banach  spaces $X$ and $Y$ and hence are $\catB$-isomorphisms. 
 \end{proof}

\begin{corollary} For any almost $\FI_K$-universal $K$-based Banach spaces
 $X$ and $Y$ and any $ \eps >0$ there exists an $\eps$-isometry $\catB$-isomorphism $f:X \rightarrow Y$.
\end{corollary}

\begin{theorem}\label{wazne} Let $U$ be an almost $\FI_K$-universal $K$-based Banach space. For any $\eps>0$ and any $K$-based Banach space $X$ there exists  an $\eps$-isometry $\catB$-morphism $f:X\to U$. \end{theorem}

\begin{proof}
Write $X$ as the completion of the union $\bigcup_{n\in \omega} X_n$ of a chain of finite dimensional $K$-based Banach subspaces $X_n$ of $X$ such that $X_0=\{0\}$. Fix a positive real number $\eps$ and choose any $\delta <\eps$.
We shall define inductively a sequence of $\delta$-isometry $\catB$-morphisms $(f_{k}:X_{k}\rightarrow  U)_{k=0}^\infty$ such that  $f_k\restriction X_{k-1}=f_{k-1}$ for every $k>0$.

We set $f_0 = 0$. Suppose that for some $k\in \omega$ a $\delta $-isometry $\catB$-morphism $f_k:X_{k}\to U$ has already been constructed.
Using the definition of the almost $\FI_K$-universality of the space $U$, we can find a $\delta $-isometry $\catB$-morphism $f_{k+1}:X_{k+1}\to U$ such that  $f_{k+1}\restriction X_{k}=f_{k}$. This completes the inductive step.

After completing the inductive construction consider the $\delta $-isometry $f:\bigcup_{k\in\w} X_k \to U$ such that  $f\restriction X_{k}=f_k$ for every $k\in \omega$.

By the uniform continuity, the $\delta $-isometry $f$ extends to an $\eps$-isometry $\bar{f}:X\to U$ such that $$\textstyle{f(\bas_X)=f\big(\bigcup_{k\in\w}B_{X_k}\big)=\bigcup_{k\in\w}f(B_{X_k})=\bigcup_{k\in\w}f_k(B_{X_k})\subset\bas_{\U},}$$ which means that $f$ is a $\catB_K$-morphism.
\end{proof}

\begin{corollary}\label{c:p} Each almost $\FI_K$-universal $K$-based Banach space $U$ is $\catB$-universal.
\end{corollary}

\begin{proof} Given a based Banach space $X$, we need to prove that $X$ is $\catB$-isomorphic to a based subspace of $U$. Denote by $X_1$ the based Banach space $X$ endowed with the equivalent norm
$$\|x\|_1=\sup_{F\subset\bas_X}\|\pr_F(x)\|.$$
It is easy to check that $X_1$ is a 1-based Banach space. By Theorem~\ref{wazne}, for $\e=\frac12$ there exists an $\e$-isometry $\catB$-morphism $f:X_1\to U$. Then $f$ is a $\catB$-isomorphism between  $X$ and the based subspace $f(X)=f(X_1)$ of the based Banach space $U$.
\end{proof}

Corollary~\ref{c:p} combined with the Uniqueness Theorem \ref{t:up} of Pe\l czy\'nski implies

\begin{corollary}\label{c:au} Each almost $\FI_K$-universal $K$-based Banach space $U_K$ is $\catB$-isomorphic to the $\catB$-universal space $\U$ of Pe\l czy\'nski.
\end{corollary}

Combining Corollary~\ref{c:au} with Theorem~\ref{t:ru=>au}, we get another model of the $\catB$-universal Pe\l czy\'nski's space $\U$.

\begin{corollary}\label{c:ru} Each $\RI_K$-universal rational $K$-based Banach space $\IU_K$ is $\catB$-isomorphic to the $\catB$-universal Pe\l czy\'nski's space $\U$.
\end{corollary}

\section{Acknowledgements} The authors express their sincere thanks to the anonumous referee for careful reading the manuscript and many valuable remarks resulting in an essential improvement of the results and their presentation.
ALso we thank the referee of the paper \cite{ataras} who noticed a gap in the proof of the initial version of Theorem~\ref{t:ru=>au}, which is now written in the correct form.

\end{document}